\documentclass[leqno,11pt]{amsart}
\usepackage{amsmath,amssymb,amsfonts,amscd}
\usepackage[mathscr]{eucal}
\usepackage{tikz}
\usepackage{verbatim}
\usepackage[colorlinks,pagebackref]{hyperref}
\usepackage{rotating}

\usepackage{color}
\theoremstyle{plain}
\newtheorem{thm}{Theorem}[section]

\newtheorem{cor}[thm]{Corollary}
\newtheorem{prop}[thm]{Proposition}

\newtheorem*{thm*}{Theorem}

\newtheorem{sublem}[equation]{Lemma}

\theoremstyle{definition}

\newtheorem{cosa}[thm]{}

\newtheorem{ex}[thm]{Example}
\newtheorem{exs}[thm]{Examples}
\newtheorem{subex}[equation]{Example}

\newtheorem{subdef}[thm]{Definition}

\theoremstyle{remark}
\newtheorem*{rem}{Remark}

\newtheorem{subrems*}{Remarks.}

\numberwithin{equation}{thm}

\newcommand{\D}{\boldsymbol{\mathsf{D}}}

\newcommand{\LL}{\mathsf L}
\newcommand{\R}{\mathsf R}
\newcommand{\N}{\operatorname{N}}

\newcommand{\tr}{{\triangleright}}

\newcommand{\sA}{\mathscr{A}}

\newcommand{\sC}{\mathscr{C}}
\newcommand{\sD}{\mathscr{D}}
\newcommand{\sE}{\mathscr{E}}

\newcommand{\sH}{\mathscr{H}}

\newcommand{\sM}{\mathscr{M}}
\newcommand{\cO}{\mathscr{O}}

\newcommand{\sS}{\mathscr{Sh2}}

\newcommand{\sU}{\mathscr{U}}
\newcommand{\sW}{\mathscr{W}}

\newcommand{\op}{{\mathsf o\mathsf p}}

\newcommand{\ZZ}{\mathbb Z}

\newcommand{\ic}{$\infty$-category}
\newcommand{\ics}{$\infty$-categories}

\newcommand{\CA}{\mathcal A}

\newcommand{\CH}{\mathcal H}

\newcommand{\bpic}{\begin{tikzpicture}}
\newcommand{\epic}{\end{tikzpicture}}

\newcommand{\Otimes}[1]{\otimes^\LL_{#1}}

\newcommand{\sHom}{\CH om}

\newcommand{\eqv}{\xto{\lift.45,\,\approx\,,}}

\newcommand{\set}{\!:=}

\newcommand{\sst}{\scriptstyle}

\newcommand{\smallcirc}{{\>\>\lift1,\sst{\circ},\,}}

\newcommand{\<}{\mkern-1mu}
\renewcommand{\>}{\mkern1mu}
\newcommand{\va}[1]{\vspace{#1pt}}

\newcommand{\halfsize}[1]{{\scalebox{.6}{#1}}}

\newcommand{\kf}{\kern.5pt}

\def\lift#1,#2,{\vbox to 0pt{\vskip-#1 ex\hbox{$\scriptstyle #2$}\vss}}

\newcommand{\OW}{\mathcal O_W}

\newcommand{\fundamentalclassa}[1]{{\boldsymbol{\mathsf{a}}}_{#1}}

\newcommand{\fundamentalclassb}[1]{{\boldsymbol{\mathsf{b}}}_{#1}}

\newcommand{\Map}[1]{\operatorname{Map}_{#1}}
\newcommand{\Fun}[2]{\operatorname{Fun}(#1,#2)}

\newcommand{\lmod}[1]{{\operatorname{LMod}}_{#1}}

\newcommand{\Bmod}[2]{#1{\textup{\tt \#}}#2}

\newcommand{\lto}{\longrightarrow}
\newcommand{\xto}{\xrightarrow}

\newcommand{\xot}{\xleftarrow}

\newcommand\iso{{\mkern8mu\longrightarrow \mkern-25.5mu{}^\sim\mkern17mu}}
\newcommand\osi{{\mkern8mu\longleftarrow \mkern-25.5mu{}^\sim\mkern17mu}}

\DeclareMathOperator{\h}{\textup{h}\!}

\DeclareMathOperator{\spec}{Spec}
\DeclareMathOperator{\Hom}{Hom}
\DeclareMathOperator{\tor}{Tor}
\DeclareMathOperator{\ext}{Ext}

\DeclareMathOperator{\id}{id}

\DeclareMathOperator{\Alg}{Alg}

\newcommand{\env}[1]{{#1}^{\mathsf e}}
\newcommand{\ov}[1]{\overline{#1}}

\def\cA #1; #2;{\cite[p.\,#1, #2]{A}}
\def\cT #1; #2;{\cite[p.\,#1, #2]{T}}

\def\lift#1,#2,{\vbox to 0pt{\vskip-#1 ex\hbox{$\scriptstyle #2$}\vss}}

\def\drlm#1{\underset{\vtop{\vskip-4.2pt\hbox to 14pt{\rightarrowfill} \vskip-10pt\hbox{$\scriptstyle \ #1$}}}\to\lim\,}

\def\dirlm#1{\lim\hskip-1.65em\lower1.37ex
       \hbox{\smash[b]{$
                   \underset{\lift 1.37,
                                         {\hbox to 0pt{\hss$\scriptscriptstyle#1$\hss}},
                                  }
                     {\:\hbox to 1.37em {\rightarrowfill}}
               $} }                      
     \!\<}

\begin{document}

\title[Adjoint associativity: an invitation to algebra in $\infty$-categories.]{Adjoint associativity: an invitation to algebra in $\infty$-categories.}

\author[J.\,Lipman]{Joseph Lipman} 
\address{Department of Mathematics,  Purdue University, West Lafayette IN 47907, U.S.A.}
\email{jlipman@purdue.edu}

\thanks{This expository article is an elaboration of a colloquium talk given April 17,~2013 at the Math.\ Sci.\ Research Institute in Berkeley,  under the Commutative Algebra program supported by the National Science Foundation.}
  
\keywords{Infinity-category algebra, adjoint associativity, Grothendieck duality, 
Hochschild derived functors,  relative perfection, relative dualizing complex.}

\subjclass[2010]{Primary: 14F05, 13D09. Secondary:  13D03}

\date{\today}

\begin{abstract}  There appeared not long ago a Reduction Formula for derived Hochschild cohomology, 
that has been useful e.g., in the study of Gorenstein maps and of rigidity w.r.t.\ semidualizing complexes.
The formula involves the relative dualizing complex of a  ring homomor\-phism, so brings out a connection between
Hochschild homology and Grothendieck~\mbox{duality}. The proof, somewhat ad hoc, uses homotopical considerations via
a number of noncanonical  projective and injective resolutions of differential graded objects.
Recent efforts aim at more intrinsic approaches, hopefully upgradable to
``higher" contexts---like bimodules over algebras in $\infty$-categories. This would lead to wider applicability,
for example to ring spectra; and the methods might be globalizable, revealing some homotopical generalizations
of aspects of Grothendieck duality. (The~original formula has a geometric Êversion, proved by completely different methods coming from duality theory.) A first step is to
extend Hom\kf-Tensor adjunction---adjoint associativity---to the $\infty$-category setting.
\va{-8}
 \end{abstract}

\maketitle

\tableofcontents

\section*{Introduction} 
There are substantial overlaps between algebra and homotopy theory, making for mutual enrichment---better understanding of some topics, and wider applicability of results from both areas.  In this vein, works of Quillen, Neeman, Avramov-Halperin, Schwede\kf-Shipley,
Dwyer-Iyengar-Greenlees, to mention just a few, come to mind. See also \cite{Gr}. In recent years, homotopy theorists like May, To\"en, Joyal, Lurie (again to mention just a few) have been developing a huge theory of \emph{algebra in $\infty$-categories,} dubbed by Lurie ``Higher Algebra",%
\footnote{Not to be confused with the contents of  \cite{HK}.}  
familiarity\- with which could be of significant benefit to (lower?)~algebraists.

This little sales pitch will be illustrated here by one specific topic
that arose algebraically, but can likely be illuminated by homotopical ideas.

\section{Motivation: Reduction of Hochschild (co)homology}
Let $R$ be a noetherian commutative ring, $\D(R)$ the derived category of the category of $R$-modules, and similarly for $S$. Let $\sigma\colon R\to S$ be an essentially-finite-type \emph{flat} homomorphism. Set $\env S\set S\otimes_R S$. Let $M,N\in\D(S)$, with $M$ \emph{$\sigma$-perfect}, i.e., the cohomology modules $H^i(M)$ are finitely generated over~$S$, and the natural image of $M$ in $\D(R)$ is isomorphic to a bounded complex of flat $R$-modules.

\begin{thm}[Reduction theorem, {\cite[Thms.\,1 and 4.6]{AILN}}]\label{RednThm} There exists a complex\/ $D^\sigma\in\D(S)$ together with bifunctorial\/ $S$-isomorphisms
\begin{align}
\R\<\<\Hom_{\env S}(S, M\Otimes{R}N)&\iso 
    \R\Hom_S\!\big(\R\<\<\Hom_S(M, D^\sigma),\>N\big),\label{RednCo}\\
S\Otimes{\env S}\<\R\<\<\Hom_R(M,N)&\iso 
    \R\<\<\Hom_S(M,D^\sigma)\Otimes{S} N.\label{RednHo}
\end{align}
\end{thm}

\noindent\emph{Remarks} 1. ``Reduction" refers to the reduction, via \eqref{RednCo} and~\eqref{RednHo}, of constructions over $\env S$ to constructions over $S$.\va1

2. The homology $S$-modules of the sources of \eqref{RednCo} and~\eqref{RednHo} are the Hochschild
cohomology modules of $\sigma$, with coefficients in \mbox{$M\Otimes{R}N$}, and the~Hochs\-child homology modules of $\sigma$,  with
coefficients in $\R\<\<\Hom_R(M,N)$,\ respectively.\va1

3. (Applications.) The isomorphism \eqref{RednCo} is used to formulate a notion of rigidity with respect to a fixed semidualizing complex \cite[\S3]{AIL}, leading to a broad generalization of the work of Yekutieli and Zhang summarized in~\cite{Y}.\looseness=-1
 
 The special case $M=N=S$ of \eqref{RednCo} plays a crucial role in the proofs of
 \cite[Theorems 3 and 4]{AI}.
 
 The special case $M=N=S$ of \eqref{RednHo} is used in a particularly simple expression for the fundamental class of $\sigma$, see \cite[Thm.\,4.2.4]{ILN}.\va1

4. The complex $D^\sigma$ is determined up to isomorphism by either \eqref{RednCo}
---which implies that $D^\sigma$ corepresents the endofunctor 
$\R\<\<\Hom_{S^e}(S, S\otimes_R -)$ of $\D(S)$---or, more directly, 
by \eqref{RednHo}---which yields an isomorphism
\[
S\Otimes{\env S}\<\R\<\<\Hom_R(S,S)\iso D^\sigma\<.
\]
 
In fact,  if $g$ is the map $\spec(\sigma)$ from $V\set\spec(S)$ to $W\set\spec(R)$,
then 
\[D^\sigma\cong g^!\OW,
\] 
i.e., $D^\sigma$ is a \emph{relative dualizing complex for $\sigma$}  \cite[Remark 6.2]{AILN}. 
    
Thus we have a relation (one of several) between Hochschild homology and Grothendieck duality.\va1
  
5. For example, if $\spec(S)$ is connected and $\sigma$ is formally smooth, so that,
with $I$ the kernel of the multiplication map $\env S\to S$,
the relative differential module $\Omega_\sigma\set I/I^2$ is locally free of constant rank, 
say~$d$, then 
\[
D^\sigma\cong \Omega^d_\sigma[d\>]\set\big(\!\wedge^{\<d}\<I/I^2\>\big)[d\>]
\cong\tor^{\env S}_d\<\<(S,S)[d\>].
\]
Using local resolutions of $S$ by Koszul complexes of  $\env S$-regular sequences that generate $I\<$, one finds a chain of natural $\D(S)$-isomorphisms 
\begin{align*}
\R\<\<\Hom_{\env S}(S,\env S)&\iso 
\big(H^d\R\<\<\Hom_{\env S}(S,\env S)\big)[-d\>]\\
&\iso
\big(H^d\R\<\<\Hom_{\env S}(S,\env S)\big)[-d\>]\otimes_{\env S} S\\
&\iso \big(H^d\big(\R\<\<\Hom_{\env S}(S,\env S)\Otimes{\env S} S\big)\big)[-d\>]\\
&\iso \big(H^d\big(\R\<\<\Hom_S(S\Otimes{\env S} S, S)\big)\big)[-d\>]\\
&\iso \<\<\Hom_S\!\big(\<\tor^{\env S}_d\<\<(S,S)[d\>],S\big)\\
&\iso\<\<\Hom_S(D^\sigma\<,S)\\
&\iso\R\<\<\Hom_S(D^\sigma\<,S).
\end{align*}
The composition $\phi$ of this chain is \eqref{RednCo} with $M=N=S$. (It is essentially the same as the isomorphism
$H^d\R\<\<\Hom_{\env S}(S,\env S)\!\iso\!\Hom_S(\wedge^{\<d}I\</I^{\>2}\<,S)$ given by  the ``fundamental local 
isomorphism" of \cite[Chap.\,III, \S7\kf\kf]{RD}.)

As $\sigma$ is formally smooth, 
$\sigma$-perfection of $M$ is equivalent to $M$ being a perfect $\env S$-complex;
and $S$ too is a perfect $\env S$-complex.
It is then straightforward to obtain \eqref{RednCo} by applying to~$\phi$ the functor 
\[
-\Otimes{\env S}(M\Otimes{R} N)=-\Otimes{S}S\Otimes{\env S}(M\Otimes{R} N)
\cong -\Otimes{S}(M\Otimes{S} N).
\]  
(Note that $M$ and $N$ may be assumed to be K-flat
over $S$, hence over $R$.)\va1
\va2

To prove \eqref{RednCo} for arbitrary $\sigma$ one uses a factorization
\[
\sigma=\textup{(surjection)}\smallcirc\textup{(formally smooth)}
\]  to reduce to the preceding formally~smooth case. For this reduction (which is the main difficulty in the proof), as well as a scheme\kf-theoretic version of Theorem~\ref{RednThm}, see \cite{AILN} and 
\cite[Theorem 4.1.8]{ILN}.

\section{Enter homotopy}
So far no homotopical ideas have appeared.  But they become necessary, via 
(graded-commutative) differential graded algebras (dgas), when the flatness assumption on $\sigma$ is dropped. Then for Theorem~\ref{RednThm} to
hold, one must first define $S^e$ to be a derived tensor product:
\[
\env S\set S\Otimes{R} S\set \ov S\otimes_R \ov S,
\]
where $\ov S\to S$ is a homomorphism of dg $R$-algebras that induces homology isomorphisms, with $\ov S$ \emph{flat} over~$R$. Such ``flat dg algebra resolutions" of the $R$-algebra $S$ exist; and any two are ``dominated" by a third. (This is well-known; for more details, see \cite[\S2, \S3]{AILN}.)
Thus $\env S$ is not an $R$-algebra, but rather a class of quasi-isomorphic dg $R$-algebras.

By using suitable ``semiprojective" dg $\ov S$-resolutions of the complexes $M$ and $N$, one can make sense of the statements 
\[
M\Otimes{R} N\in\D(\env S), \qquad\R\<\<\Hom_R(M,N)\in\D(\env S);
\]
and then, following Quillen, Mac Lane and Shukla, define  complexes
\[
\R\<\<\Hom_{\env S}(S, M\Otimes{R}N),\qquad S\Otimes{\env S} \R\<\<\Hom_R(M,N)
\]
whose homology modules are the
\emph{derived Hochschild cohomology} resp.~\emph{the derived Hochschild homology} 
modules of $\sigma$, with coefficients in $M\Otimes{R}N$ and~$\R\<\<\Hom_R(M,N)$,
respectively. These complexes depend on a number of choices of resolution, 
so they are defined only up to a coherent family of
isomorphisms, indexed by the choices. This is analogous to what happens when one
works with derived categories of modules over a ring.

In \cite{AILN}, the reduction of Theorem~\ref{RednThm} to the formally smooth case is done by the manipulation of a number of noncanonical dg resolutions, both semiprojective and semiinjective. Such an argument tends to obscure the conceptual structure. 
Furthermore, in section 6 of that paper a geometric version of Theorem~\ref{RednThm} is proved by completely different methods associated with Grothendieck
duality theory---but only for flat maps.  A globalized theory of derived Hochschild (co)homology for analytic spaces or noetherian schemes of characteristic zero is given, e.g., in \cite{BF}; but there is as yet no extension of Theorem~\ref{RednThm} to nonflat maps of such spaces or schemes.\looseness=-1

The theory of algebra in $\infty$-categories, and its globalization ``derived algebra geometry," encompass all of the above situations,%
\footnote{to some extent, at least: see e.g., \cite{Sh2}.   But see also~\ref{derived tensor} and~\ref{derived hom} below.} 
 and numerous others, for instance
``structured spectra" from homotopy theory.  The (unrealized) underlying goal toward which this lecture is a first step is to prove
a version of Theorem~\ref{RednThm}---without flatness hypotheses---that is meaningful in this general context.  The hope is that such a proof could unify the
local and global versions in \cite{AILN}, leading to better understanding and wider applicability; and perhaps most importantly, to new insights into, and generalizations of, Grothendieck duality.\looseness=-1

\section{Adjoint Associativity}\label{AA3}
To begin with, such an upgraded version of Theorem~\ref{RednThm} must involve some generalization of $\otimes$ and Hom;
and any proof will most probably involve the basic relation between these functors,
namely \emph{adjoint associativity}.

For any two rings (not necessarily commutative) $R$, $S$, let $R\#S$ be the abelian category 
of $R\>$-$S$ bimodules ($R$ acting on the left and $S$ on the right).

The classical version of adjoint associativity (cf.~\cite[VI, (8.7)]{M}) asserts that
\emph{for rings $A$, $B$, $C$, $D$, and
$x\in A\#B$, $y\in B\#C$, $z\in D\#C$, there exists in\/ $D\#A$ a functorial isomorphism}\begin{equation*}\label{classical}
a(x,y,z)\colon\Hom_C(x\otimes_B \<y, z)\iso \Hom_B(x,\Hom_C(y,z))
\tag{\ref{AA3}.1}
\end{equation*}
\emph{such that for any fixed\/ $x$ and\/ $y,$ the corresponding isomorphism between
the left adjoints of the target and source of\/ $a$ is the associativity isomorphism} 
\[
-\otimes_A (x\otimes_B \<y)\osi (-\otimes_A x)\otimes_B y.
\]

As hinted at above, to get an analogous statement for \emph{derived categories,} where one needs flat resolutions to define (derived) tensor~products,
one has to work in the dg world; and that suggests going all the way to 
\emph{\ics}.

The remainder of this talk will be an attempt to throw some light on
how~\eqref{classical} can be formulated and proved in the $\infty$-context. There will be no possibility of
getting into details, for which however liberal references will be given to the massive works \cite{T} and \cite{A} (downloadable from Lurie's home page
{\tt www.math.harvard.edu/\~{}lurie/\kf}), for those who might be prompted to explore the subject matter more thoroughly.%
\footnote{The page numbers in the references to \cite{A} refer to the preprint dated August, 2012.}

\section{\ics} It's time to say what an \ic\ is. 

An ordinary small category $\sC$ is, to begin with, a diagram
\[
 \bpic[xscale=3, yscale=2]

  \node(11) at (.977,-1){$A_1$};
  \node(13) at (2.023,-1){$A_0$};   
  \node(12) at (1.5,-1)[scale=.75]  {$s_{\halfsize0}$};
  \draw[->] (12)--(11);
  \draw[-]   (12)--(13);
  \draw[->] (1.1,-.87)--(1.9,-.877) node[above=1pt, midway, scale=.75]{$d_1$};
  \draw[->] (1.1,-1.13)--(1.9,-1.13) node[below=1pt, midway, scale=.75]{$d_{\halfsize0}$};

 \epic
\]
where $A_1$ is the set of arrows in $\sC$,  $A_0$ is the set of objects, 
$s_0$ takes an object to its identity map, and $d_0$ (resp.~$d_1$) takes an arrow to  its target (resp.~source). We can extend this picture by introducing sequences of composable arrows:
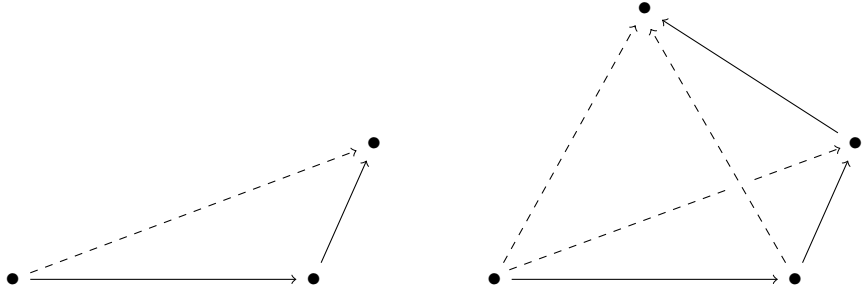
\[
 \bpic[xscale=2, yscale=1.8]
  \node(12) at (3.4,-1){$\bullet$};

  \node(21) at (1,-2){$\bullet$};
  \node(23) at (3,-2){$\bullet$};
  \node(25) at (4.2,-2){$\bullet$};
  \node(27) at (6.2,-2){$\bullet$};

  \node(16) at (6.6,-1){$\bullet$};

  \node(06) at (5.2,0){$\bullet$};

  \draw[->](21)--(23);
  \draw[->,dashed] (21)--(12);
  \draw[->] (23)--(12);

  \draw[->] (25)--(27);
  \draw[->] (27)--(16);
  \draw[->] (16)--(06);
  
  \draw[-,dashed](4.3,-1.93)--(5.75,-1.337);
  \draw[->,dashed](5.93,-1.263)--(6.5,-1.03);
  \draw[->,dashed](25)--(06);
  \draw[->,dashed](27)--(5.24,-.15);
 \epic
\]
The first picture represents a sequence of two composable arrows, whose composition is represented by the dashed arrow; and the second picture represents
a sequence of three composable arrows, with dotted arrows representing compositions of two or three of these.  The pictures suggest calling a sequence of $n$ composable arrows an \emph{$n$-simplex}. (A 0\kf-simplex is simply an object in $\sC$.) The set of 
$n$-simplices is denoted $A_n$.

There are four \emph{face maps} $d_i\colon A_3\to A_2\ (0\le i\le 3)$, taking a sequence
$\gamma\smallcirc\beta\smallcirc\alpha$ to the respective sequences 
$\gamma\smallcirc\beta$, $\gamma\smallcirc(\beta\alpha)$,
$(\gamma\beta)\smallcirc \alpha$ and $\beta\smallcirc\alpha$.
There are three \emph{degeneracy maps} $s_j\colon A_2\to A_3\ (0\le j\le2)$ taking
a sequence $\beta\smallcirc\alpha$ to the respective ``degenerate" (that is, containing an identity map) sequences $\beta\smallcirc\<\alpha\smallcirc\!\id$, $\beta\smallcirc\!\id\!\smallcirc\alpha$ and $\id\<\<\smallcirc\beta\smallcirc\alpha$.

Likewise, for any $n>0$ there are face maps $d_i\colon A_n\to A_{n-1}\ (0\le i\le n)$
and degeneracy maps $s_j\colon A_{n-1}\to A_n\ (0\le j<n)$; and these maps satisfy the standard identities that define a \emph{simplicial set} (see e.g., \cite[p.\,4, (1.3)]{GJ}).

The simplicial set $\N(\sC)$ just defined is called the \emph{nerve of} $\sC$.\va2

\noindent\emph{Example.}  For any $n\ge 0,$ the totally ordered set of integers 
\[
0<1<2<\cdots<n-1<n
\]
can be viewed as a category (as can any ordered set). The nerve of this category is the \emph{standard $n$-simplex,}
denoted $\Delta^{\<n}$. Its $m$-simplices identify with the nondecreasing maps from the integer interval $[0,m]$ to $[0,n]$. In particular, there is a unique nondegenerate 
$n$-simplex $\iota_n$, namely the identity map of $[1,n]$.

The collection of all the nondegenerate simplices of $\Delta^{\<n}\<$, and their face maps, can be visualized by means of the usual picture of a geometric  $n$-simplex and its  subsimplices.
(For $n=2$ or 3, see the above pictures, with all dashed arrows made solid.)\va1

The \emph{horn} $\Lambda_i^n\subset \Delta^{\<n}$ is the simplicial subset  whose $m$-simplices ($m\ge 0$) are the nondecreasing maps  $s\colon [0,m]\to [0,n]$ with  image not containing  the set $\big([0,n]\setminus\{i\}\big)$. For example, there are 
$n$ nondegenerate $(n-1)$-simplices  namely $d_j\iota_n\ (0\le j\le n,\; j\ne i)$.  \va1

Visually, the nondegenerate simplices of $\Lambda_i^n\subset \Delta^{\<n}$ are those subsimplices of a geometric $n$-simplex other than the $n$-simplex itself and its $i$-th face. \va2

Small categories are the objects of a category $\sC$at whose morphisms are functors; and
simplicial sets form a category $\sS\textup{et}_\Delta$ whose morphisms are
\emph{simplicial maps}, that is,  maps taking $m$-simplices to $m$-simplices (for all $m\ge0$) and commuting with all the face and degeneracy maps. The above map $\sC\mapsto \N(\sC)$ extends in an obvious way to a \emph{nerve functor}
$\sC\textup{at} \to\sS\textup{et}_\Delta\>$.

\begin{prop}\label{nerve functor} \textup{(\cite[p.\,9, 1.1.2.2]{T}.)} The nerve functor\/ $\textup{$\sC$at} \to \sS\textup{et}_\Delta$ is a fully faithful embedding. Its essential image is the full subcategory of\/ $\sS\textup{et}_\Delta$ spanned by the simplicial sets\/ $K$ with the following property$\>\>:$

$(*)$ For all $n>0$ and $0<i<n$, every simplicial map $\Lambda_i^n\to K$ extends
\emph{uniquely} to a simplicial map\/ $\Delta^{\<n}\to K\<$.

\end{prop}

\noindent\emph{Remarks.}  By associating to each simplicial  map $\Delta^{\<n}\to K$ the image of the nondegenerate $n$-simplex $\iota_n$, one gets a bijective correspondence between such maps and $n$-simplices of $K$. (See, e.g., \cite[p.\,6]{GJ}.) 

By associating to each simplicial  map $\Lambda_i^n\to K$ the image of the sequence $(d_j\iota_n)_{0\le j\le n,\; j\ne i}$, one gets a bijective correspondence\va{.7} between such maps and sequences $(y_j)_{0\le j\le n,\; j\ne i}$ of $(n-1)$-simplices of $K$\va{.7}  such that $d_jy_k=d_{k-1}y_j$ if  $j<k $ and $j,\> k\ne  i$. (See \cite[p.\,10, Corollary 3.2]{GJ}.)

Thus $(*)$ means that \emph{for all\/ $n>0$ and\/ $0<i<n$, if\/ $\lambda$ is the map from the set of\/ $n$-simplices of\/~$K$ to the set of such sequences\/ $(y_j)$ that takes an $n$-simplex\/~$y$ to the sequence\/ 
$(d_jy)_{0\le j\le n,\; j\ne i}\>,$ then $\lambda$ is bijective}.

\begin{subdef}\label{ic}
An \emph{\ic}\ is a simplicial set\/ $K$  such that for all $n>0$ and $0<i<n$, every simplicial map $\Lambda_i^n\to K$ extends to a simplicial map\/ $\Delta^{\<n}\to K\<$, i.e., a $K$ for which the preceding map $\lambda$ is \emph{surjective}.

A \emph{functor} from one \ic\ to another is a map of simplicial sets.
\end{subdef}

Thus \ics\ and their functors form a full subcategory of~$\sS\textup{et}_\Delta$, one that itself has a full subcategory canonically isomorphic to $\sC$at.

\begin{subex}\label{complexes} 
To any dg category $\sC$ (one whose arrows between two fixed objects are complexes of abelian groups, composition being bilinear)
one can assign the \emph{dg-nerve} $\N_{\textup{dg}}(\sC)$, an \ic\ whose construction is more complicated than that of the nerve~$\N(\sC)$ because the dg structure has to be taken into account. (For details, see \cite[\S1.3.1]{A}.)

For instance, the complexes in an abelian category~$\sA$ can be made into a dg category
$\sC_{\textup{dg}}(\sA)$ by defining $\Hom(E, F)$ for any complexes $E$ and~$F$ to be the complex 
of abelian groups that is
$\Hom\!\big(E, F[n]\big)$ in degree $n$, with the usual differential.  \va{.6} When $\sA$ is a Grothendieck abelian category, we will see below (Example~\ref{derived}) how one extracts from the \ic\ $N_{\textup{dg}}(\sC_{\textup{dg}}(\sA))$ the usual derived category $\D(\sA)$.
\end{subex}

\begin{subex}\label{spaces}
 To any topological category $\sC$---that is, one where the Hom sets are topological spaces and composition is continuous---one can assign a \emph{topological nerve} $\N_{\textup{top}}(\sC)$, again more complicated than the usual nerve $N(\sC)$ \cite[p.\,22, 1.1.5.5]{T}.

CW-complexes are the objects of a topological category $\sC\sW$.  The topological nerve
$\sS\set \N_{\textup{top}}(\sC\sW)$ is an \ic,  the \emph{\ic\ of spaces}.~ 
(See \cite[p.\,24, 1.1.5.12\kf; p.\,52, 1.2.16.3]{T}.) Its role in the theory of \ics\ is analogous to the role of the category of sets in ordinary category theory. \looseness=-1
\end{subex}

\begin{subex}\label{Kan}
 \emph{Kan complexes}  are simplicial sets such that the defining condition of \ics\ holds for \emph{all} $i\in[0,n]$. Examples are \emph{the singular complex of a topological space} $X$ (a simplicial set that encodes the homotopy theory of $X$), \emph{the nerve of 
a groupoid} (=\:category with all maps  isomorphisms),  and \emph{simplicial abelian groups}. (See \cite[p.\,8, 1.1.2.1]{T} and \cite[\S I.3]{GJ}.)\looseness=-1

Kan complexes span a full subcategory of the category
of \ics, the inclusion having a right adjoint \cite[p.\,36, 1.2.5.3]{T}. The \emph{simplicial nerve} of this subcategory \cite[p.\,22, 1.1.5.5]{T} provides another model for the \ic\ of spaces \cite[p.\,51, 1.2.16]{T}.
\end{subex}

Most of the basic notions from category theory can be extended to \ics. Several examples
will be given as we proceed.  A first attempt at such an extension would be to express a property of categories in terms of their nerves, and then to see if this formulation makes sense for arbitrary \ics. (This will not always be done explicitly; but as \ic\ notions are introduced, the reader might check that when  restricted to nerves, these notions reduce to the corresponding classical ones.)

\begin{subex} An \emph{object} in an \ic\  is a 0\kf-simplex. A \emph{map} $f$ in an \ic\ is a 1-simplex. The \emph{source} (resp.~\emph{target}) of $f$ is the object $d_1f$ (resp.~$d_0f$).
The \emph{identity map} $\id_x$ of an object $x$ is the map $s_0x$, whose source and target are both $x$.
\end{subex}

\medskip

Some  history and motivation related to \ics\  can be gleaned, e.g., starting from {\tt <http://ncatlab.org/nlab/show/quasi-category>}.\va1

The notion of \ic\ as a generalization of that of category grew out of the study of operations in the homotopy category of topological spaces, for instance the composition of paths. Indeed, as will emerge, the basic effect of removing unicity from condition $(*)$ above to get to \ics\ (Definition~\ref{ic}) is to replace \emph{equality} of maps in categories with a \emph{homotopy}
relation, with all that entails.
 
 Topics of foundational importance in homotopy theory, such as \emph{model categories,} or
 \emph{spectra} and their products, are closely related to, or can be treated via, \ics\ 
 \cite[p.\,803]{T}, \cite[\S1.4, \S6.3.2]{A}.   Our concern here will mainly be with relations to algebra.

\section{The homotopy category of an \ic.}
\begin{cosa}
The nerve functor of Proposition~\ref{nerve functor} has a left adjoint $h\colon\sS\textup{et}_\Delta\to \textup{$\sC$at}$, the \emph{homotopy functor}, see \cite[p.\,28, 1.2.3.1]{T}. 

If the simplicial set  $\sC$ is an \ic, the homotopy category $h\sC$ can be constructed as follows. 
For maps $f$ and $g$ in $\sC$, write $f\sim g$ (and say  that ``$\<f$ is homotopic to $g$") if there is a 2-simplex $\sigma$ in $\sC$ such that $d_2\sigma=f$, $d_1\sigma=g$ and
$d_0\sigma=\id_{d_0\mkern-.5mu g}=\id_{d_0 f}\>$:
\[
 \bpic[xscale=2, yscale=1.8]
  \node(12) at (2,-1){$\bullet$};

  \node(21) at (1,-2){$\bullet$};
  \node(23) at (3,-2){$\bullet$};

  \draw[->](21)--(23) node[below = 1pt, midway, scale=.75]{$f$};
  \draw[->] (21)--(12) node[above=1pt, midway, scale=.75]{$g\mkern5mu$};
  \draw[->] (23)--(12) node[above, midway, scale=.75]{$\mkern25mu\id$};
  
  \node at (2,-1.6)[scale=.9]{$\sigma$};
 \epic
\]
(This can be intuited as the skeleton of a deformation of $f$ to $g$ through a ``continuous family" of maps with fixed source and target.)
Using the defining property of \ics, one shows that this homotopy relation is an equivalence relation. Denoting the class of $f$ by $\bar f$, one defines the composition 
$\bar f_2\smallcirc\! \bar f_1$ to~be $\bar h$ for any $h$ such that there exists a 2-simplex
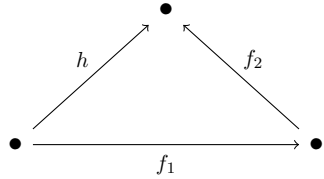
\[
 \bpic[xscale=2, yscale=1.8]
  \node(12) at (2,-1){$\bullet$};

  \node(21) at (1,-2){$\bullet$};
  \node(23) at (3,-2){$\bullet$};

  \draw[->](21)--(23) node[below = 1pt, midway, scale=.75]{$f_1$};
  \draw[->] (21)--(12) node[above=1pt, midway, scale=.75]{$h\mkern14mu$};
  \draw[->] (23)--(12) node[above, midway, scale=.75]{$\mkern20mu f_2$};

 \epic
\]
One shows that this composition operation is well-defined, and associative.
There results a category whose objects are those of $\sC$, and whose maps are the homotopy equivalence classes of maps in~$\sC$, with composition as just described. (For details, see \cite[\S1.2.3]{T}.) This is the homotopy category $h\sC$.\va2\looseness=-1
\end{cosa}

\begin{ex}\label{hspaces}
Let $\sS$ be the \ic\ of spaces (Example~\ref{spaces}). Its homotopy category
$\sH\set \h\sS$ is called the \emph{homotopy category of spaces}.
The objects of $\sH$ are CW-complexes, and the maps are homotopy-equivalence classes
of continuous maps. (See \cite[p.\,16]{T}.) 
\end{ex}

\begin{ex}[extending Example~\ref{complexes}] \label{derived}
In the category of complexes $\textup{Ch}(\CA)$ in a Grothendieck category~$\CA$,\va{.5}  the (injectively) \emph{fibrant} objects are those complexes~$I$ such~that for any $\CA$-diagram of complexes $X\xot{\lift.5,\>\>s\>,}Y\xto{\lift1.1,\>f,}I$ with $s$ both\va{.7} a (degreewise) monomorphism and a quasi-isomorphism, there exists $g\colon X\to I$ such that $gs=f$.

\begin{sublem} \label{q-inj and fibrant}
A complex\/ $I$ is q-injective $(\<\<$aka K-injective$)$ if and only if $I$~is homotopy-equivalent to a fibrant complex.
\end{sublem}

\begin{proof}
Fix a fibrant $Q$. By \cite[p.\,97, 1.3.5.11]{A},   if the complex~$M$ is exact then so is the complex~
$\Hom^\bullet(M,Q)$; and by \cite[2.3.8(iv) and (2.3.8.1)]{li}, this means that $Q$ is q\kf-injective,
whence so is any complex homotopy-equivalent to $Q$.

If $I$ is q\kf-injective, then factoring $I\to 0$ as fibration$\smallcirc$(trivial cofibration) (\cite[p.\,93, 1.3.5.3]{A}) one gets a 
monomorphic quasi-isomorphism~\mbox{$j\colon I\hookrightarrow Q$} with $Q$ fibrant, hence q\kf-injective; so  $j$  is a homotopy equivalence \cite[2.3.2.2]{li}. 
\end{proof}

\begin{subrems*}
If the complex $Q$ is bounded below and injective in each degree, then $Q$ is fibrant, hence q\kf-injective, see \cite[p.\,96, 1.3.5.6]{A}.

2. Any split short exact sequence, extended infinitely both ways by zeros, is a complex
homotopically equivalent to the fibrant complex 0, but not necessarily itself fibrant, since fibrant complexes
are termwise injective, see again \cite[p.\,96, 1.3.5.6]{A}.
\end{subrems*}

Next, for any additive category $\CA$, two maps in the dg-nerve $\N_{\textup{dg}}(\textup{Ch}(\CA))$ are homotopic iff they are so as chain maps, see \cite[p.\,64, 1.3.1.8]{A}.  Thus the homotopy category h$\N_{\textup{dg}}(\textup{Ch}(\CA))$ is just 
the category  whose objects are the $\CA$-complexes and maps are 
homotopy-equivalence classes of chain\va2  maps.\looseness=-1

\pagebreak[3]
Similarly, when $\CA$ is a Grothendieck abelian category and $\textup{Ch}(\CA)^0$
is the full subcategory of $\textup{Ch}(\CA))$ spanned by the fibrant complexes,  the homotopy\- category of the derived $\infty$-category\va{.6} 
${\sD}(\CA)\set \N_{\textup{dg}}(\textup{Ch}(\CA)^0\>)$  \cite[p.\,96, 1.3.5.8 and p.\,65, 1.3.1.11]{A} is the quotient of\va{.3} 
$\textup{Ch}(\CA)^0$ by the homotopy relation on chain maps, and thus is equivalent to the similar
category whose objects are the fibrant compexes, which by \ref{q-inj and fibrant} is equivalent to the usual derived category $\D(\CA)$.  
 
 In summary: $h\sD(\CA)$ \emph{is equivalent to} $\D(\CA)$.\va2

(A more general result for any dg category is in \cite[p.\,64, 1.3.1.11]{A}.) 

\end{ex}

\begin{rem}
The homotopy category of a \emph{stable} \ic\ is triangu\-lated. (See Introduction to \cite[\S1.1]{A}.)
For instance, the \ic\ ${\sD}(\CA)$ (just above) is stable \cite[p.\,96, 1.3.5.9]{A}.
So ~is the \ic\ of \emph{spectra}---whose homotopy category underlies stable homotopy theory \cite[p.\,16, 1.1.1.11]{A}. 
\end{rem}

\begin{ex}\label{model cat} A localization $\sD\to\sD_V$ of an ordinary 
category~$\sD$ w.r.t a set~$V$ of maps  in $\sD$ is an initial object in the category of functors with source~$D$ that take the maps in $V$ to isomorphisms. 

A  localization $\sC\to\sC[W^{-1}]$ of an \ic\ $\sC$ w.r.t a set $W$ of maps (i.e., 1-simplices) in $\sC$ is similarly universal \emph{up to homotopy} for those \mbox{$\infty$-functors} out of $\sC$ that take the maps in $W$ to equivalences. (For more precision, see \cite[p.\,83, 1.3.4.1]{A}.) Such a localization exists, and is  determined uniquely up to equivalence by $\sC$ and $W$ \cite[p.\,83, 1.3.4.2]{A}. 

For functors of the form $\sC\to\N(\sD)$ with $\sD$ an ordinary category, 
the words ``up to homotopy" in the preceding paragraph can be omitted.
(This follows from the precise definition of localization, because in 
\ics\ of the form Fun$(\sC, \N(\sD))$---see \S\ref{Fun}---the only equivalences are identity maps.)

So  composition with the localization map (see \eqref{comp}) gives a natural bijection from the set of $\infty$-functors
$\sC[W^{-1}]\to{}$N$(\sD)$ to the set of those $\infty$-functors $\sC\to{}$N$(\sD)$ that take the maps in $W$ to equivalences, that is, from the set of functors
h$(\sC[W^{-1}])\to\sD$ to the set of those functors h$\sC\to\sD$ that take the maps in the image $\bar W$ of~$W$ to isomorphisms. Hence there is a natural isomorphism
\begin{equation}\label{homotopy and localization}
\h\big(\sC[W^{-1}]\big)\iso \big(\<\<\h\sC\big)_{\bar W}
\end{equation}
giving commutativity of the homotopy functor with localization. \va2

%
%
%
%
%
%

For instance, to every \emph{model category}~{\bf A} one can associate naturally 
an ``underlying \ic"~$\mathbf A_\infty$.  Under mild assumptions, $\mathbf A_\infty$ can be taken to be the localization (N$(\mathbf A))[W^{-1}]$,  with $W$  the set of weak equivalences 
in~{\bf A}. Without these assumptions, one can just replace $\mathbf A$ by its full subcategory spanned by the cofibrant objects \cite[p.\,89, 1.3.4.16]{A}.

Equation~\ref{homotopy and localization} shows
 that the homotopy category h$\mathbf A_\infty$ is canonically isomorphic to $\mathbf A_{\bar W}$, the classical homotopy category of~{\bf A} \cite[p.\,75, Thm.\,1.11]{GJ}.

Example~\ref{derived}, with $\CA$  the category of 
right modules over a fixed ring~$R$, is essentially the case where $\mathbf A$ is the category of
complexes in~$\CA$, with  ``\kf projective"~model structure as in \cite[p.\,814, 8.1.2.8]{A}.
\end{ex}

\begin{cosa}\label{map space}
An important feature of \ics\ is that any two objects determine not just the set of maps from one to the other, but also a topological \emph{mapping~space}. In fact, with $\sH$ as in~\ref{hspaces}, the homotopy category $\h\sC$ of an \ic\ $\sC$ can be upgraded to an \emph{$\sH$-enriched} category, as follows: 

For any objects $x$ and $y$ in $\sC$, one considers not only maps with source~
$x$ and target $y$, but all ``arcs"  of $n$-simplexes $(n\ge 0)$ that go from  the trivial $n$-simplex $\Delta^{\<n}\to\Delta^{\<0}\xto{x\>} \sC$ to the trivial $n$-simplex $\Delta^{\<n}\to\Delta^{\<0}\xto{y\>} \sC$---more precisely,  maps
$\theta\colon\Delta^{\<1}\times\Delta^{\<n}\to\sC$ such that the compositions
\[
\Delta^{\<n}=\Delta^{\<0}\times\Delta^{\<n}\xto{\<j\times\id\>\>}\Delta^{\<1}\times\Delta^{\<n}
\xto{\ \theta\ }\sC
\qquad(j=0,1)
\]
are the unique $n$-simplices in the constant simplicial sets $\{x\}$ and $\{y\}$ respectively.
Such $\theta$ are the $n$-simplices of a Kan subcomplex $M_{x,y}$ of the ``function complex"
$\mathbf{Hom}(\Delta^{\<1}, \sC)$ \cite[\S I.5]{GJ}. The \emph{mapping space} $\Map\sC(x,y)$
is the geometric realization of $M_{x,y}$. It is a CW-complex \cite[\S I.2]{GJ}. For objects $x,y,z\in\sC$, there is in $\sH$ a composition map
\[
\Map\sC(y,z)\times \Map\sC(x,y)\lto\Map\sC(x,z),
\]
that, unfortunately, is not readily describable (see \cite[pp.\,27--28, 1.2.2.4, 1.2.2.5]{T}); and this composition satisfies associativity.

The unenriched homotopy category is the underlying ordinary category, obtained by replacing
each $\Map\sC(x,y)$ by the set $\pi_0\Map\sC(x,y)$ of its connected components.
\end{cosa}

\begin{ex}\label{pointless}
When $\sC=\N(C)$ for an ordinary category $C\<$, the preceding discussion is pointless: the spaces $\Map\sC(x,y)$ are isomorphic in $\sH$ to discrete topological spaces 
(see, e.g., \cite[p.\,22, 1.1.5.8; p.\,25, 1.1.5.13]{T}), so that the $\sH$-enhancement of $\h\N(C)$ is trivial;  and one checks that the
counit map is an isomorphism of ordinary categories $\h\N(C)\iso C$. 

More generally---and much deeper, for any topological category $\sD$ and any simplicial category $\sC$ that is \emph{fibrant}---that is,  all its mapping complexes are Kan complexes, one has natural $\sH$-enriched isomorphisms
\begin{align*}
\h\N_{\textup{top}}(\sD)\iso \h\sD\textup{\ \ resp.\ \ }  \h\N_\Delta(\sC)\iso \h\sC
\end{align*}
where $\N_{\textup{top}}(\sD)$ is the topological nerve of $\sD$ (an \ic\ \cite[p.\,24, 1.1.5.12]{T}) and
$\N_\Delta(\sC$) is the simplicial nerve of $\sC$ (an \ic\ \cite[p.\,23, 1.1.5.10]{T}), where  the topological homotopy category $\h\sD$ is obtained from~$\sD$ by replacing each topological space $\Map{\sD}(x,y)$ by a weakly homotopically equivalent CW complex considered as an object of $\sH$ \cite[p.\,16, 1.1.3.4]{T}, 
and the simplicial homotopy\- 
category~$\h\sC$ is obtained from $\sC$ by replacing each simplicial set $\Map{\sC}(x,y)$ by its geometric realization considered as an object of $\sH$ (see \cite[p.\,19]{T}). Using the description of the homotopy category of a simplicial set given in \cite[p.\,25, 1.1.5.14]{T}, one finds that
the first isomorphism is essentially \cite[p.\,25, 1.1.5.13]{T}; and likewise, the second is essentially
\cite[p.\,72, 2.2.0.1]{T}.
\end{ex}

\begin{ex}\label{homotopy gp} Let $\sD$ be a dg category.  For any two objects $x,y\in\sD$ replace the mapping complex
$\Map{\sD}(x,y)$ by the simplicial abelian group associated by the Dold-Kan correspondence to the truncated complex $\tau_{\halfsize{$\le0$}}\Map{\sD}(x,y)$, to produce a 
simplicial category~$\sD_{\<\<\Delta}\>$. (See \cite[p.\,65, 1.3.1.13]{A}, except that indexing here is cohomological rather than homological.)

For example, with notation as in ~\ref{derived}, 
$\D(\CA)$ is also the homotopy category of  the simplicial nerve of 
the simplicial category thus associated to $\textup{Ch}(\CA)^0$ \cite[p.\,66, 1.3.1.17\kf\kf]{A}.

For the category  $\CA$ of abelian groups, and $\sD\set\textup{Ch}(\CA)^0$,  
{\cite[p.\,47, Remark 1.2.3.14]{A}} (in light of \cite[p.\,46, 1.2.3.13]{A}) points to an agreeable interpretation of the~homotopy groups  of the Kan complex
$\textup{Map}_{{\sD}_{\!\Delta}}\!(x,y)$\va{.6} with base point 0 (or of its geometric realization, see \cite[bottom, p.\,60]{GJ}):
\[
\pi_n(\textup{Map}_{{\sD}_{\!\Delta}}\!(x,y))\cong\textup{H}^{-n}\<\Map{\sD}(x,y)=:\ext^{-n}(x,y)\qquad(n\ge 0).
\]
(See also  \cite[p.\,32, 1.2.1.13]{A} and \cA 29; \S1.2, 2nd paragraph;.)

\end{ex}

\begin{cosa}
A functor $F\colon \sC_1\to\sC_2$ between two \ics\ induces a functor 
$\h F\colon\h\sC_1\to\h\sC_2$ of $\sH$-enriched categories: the functor $\h F$ has the same effect on objects as $F$ does, and  there is a natural family of
 $\sH$-maps
\[
\h F_{x,y}\colon\Map{\sC_1}\<\<(x,y) \to \Map{\sC_2}\<(Fx,Fy)
\qquad(x,y\textup{ objects in }\sC_1)
\]
that respects composition (for whose existence see 
\cite[p.\,25, 1.1.5.14 and p.\,27, 1.2.2.4]{T}.)

The functor $F$ is called a \emph{categorical equivalence} if for all $x$ and $y$, $\h F_{x,y}$
is a homotopy equivalence (\,=\:isomorphism in $\sH\>$), and for every object $z\in\sC_2$,
there exists an object $x\in\sC_1$ and a map $f\colon z\to Fx$ whose image in $\h\sC_2$ is an isomorphism.

\begin{ex} For any \ic\ $\sC,$ the unit map $\sC\to\N(\h\sC)$ induces an isomorphism of ordinary homotopy categories; but it is a categorical equivalence
only when the mapping spaces of $\sC$ are isomorphic in $\sH$ to discrete topological spaces,
i.e., their connected components are all contractible.
\end{ex}

The ``interesting" properties of \ics\ are  those
which are invariant under categorical equivalence.
In other words, the $\sH$-enriched homotopy category is the fundamental invariant of an \ic\ $\sC$; the role of $\sC$ itself is to generate information about $h\sC$. 

For this purpose, 
$\sC$ can be replaced by any equivalent \ic, i.e., an \ic\ that can be joined to~$\sC$ by a chain of equivalences 
(or even by equivalent topological or simplicial categories, as explained in \cite[\S 1.1]{T}, and illustrated by Example~\ref{pointless} above). 
Analogously, one can think of a single homology theory in topology or algebra being constructed in various different ways.\va1

Along these lines, a \emph{map} in $\sC$ is called an equivalence if the induced map in $\h\sC$ is an isomorphism; and the interesting properties of  objects in $\sC$ are those which are invariant under equivalence.
\end{cosa}

\begin{ex}
An \ic\ $\sC$ is a Kan complex (\S\ref{Kan}) if and only if every map in $\sC$ is an equivalence, i.e., $h\sC$ is a groupoid \cite[\S\,1.2.5]{T}. For a Kan complex $\sC$,   $\h\sC$ is the \emph{fundamental groupoid} of~$\sC$  (or of its geometric realization), see \cite[p.\,3, 1.1.1.4]{T}.
\end{ex}

\section{Colimits}
To motivate the definition of colimits in \ics,  recall that a colimit of a functor 
$\tilde p\colon K\to C$ of ordinary categories is  an initial object in the category $C_{\tilde p/}$
whose objects are the extensions of $\tilde p$
to the right cone $K^\tr$---that is, the disjoint union of $K$ and the
trivial category $*$ (the category with just one map) together with one arrow from each object of $K$ to the unique object of $*\>$---and whose maps are the obvious ones. 

Let us now reformulate this remark in the language of \ics. (A fuller discussion appears in 
\cite[\S\S\,1.2.8, 1.2.12 and 1.2.13]{T}.)

First, an \emph{initial object} in an \ic\ $\sC$ is an object $x\in\sC$ such that 
for every object $y\in\sC$, the mapping space $\Map{\sC}(x,y)$ is contractible.
It is equivalent to say that $x$ is an initial object in 
the $\sH$-enriched homotopy category~$\h\sC$.  Thus any two initial objects in $\sC$ are equivalent. (In fact,  if nonempty, the set of initial objects in $\sC$ spans a contractible
Kan subcomplex of $\sC$ \cite[p.\,46, 1.2.12.9]{T}.)

Next, calculation of the nerve of the above right cone $K^{\>\tr}$ suggests the following definition. For \emph{any} simplicial set $K\<$, the \emph{right cone} $K^{\>\tr}$ is the simpicial set whose
set of $n$-simplices $K_n^{\tr}$ is the disjoint union of all the sets $K_m\ (m\le n)$ and   
 $\Delta_n^{\<0}$ (the latter having a single member $*_n$),
with the face maps $d_j$---when $n>0\>\>$---(resp.~degeneracy maps~$s_j$) restricting on~$K_m$ to the usual face (resp.~degeneracy) maps for $0\le j\le m$
(except that $d_0$ maps all of ~$K_0$ to~$*_{n-1}$), and to identity maps for $m<j\le n$, 
and taking $*_n$ to~$*_{n-1}$ (resp.~$*_{n+1}$). It may help here to observe that for $n>0$, the 
nondegenerate $n$-simplices in $K^{\>\tr}$ are just the nondegenerate $n$-simplices in~$K_n$ together with the nondegenerate $(n-1)$ simplices in~$K_{n-1}$, the latter visualized as being  joined to the ``vertex" $*_n\>$.  

This  is a special case of the construction (which we'll not need) of the \emph{join} of two simplicial sets \cite[\S1.2.8]{T}. The join of two \ics\ is an \ic\ \cite[p.\,41, 1.2.8.3]{T}; thus
if $K$ is an \ic\ then so is $K^{\>\tr}\<$.

Define $K^{\>\tr^{\<n}}$  inductively by $K^{\>\tr^{\<1}}\set K^\tr$ and (for $n>1$)
$K^{\>\tr^{\<n}}\set(K^{\>\tr^{\<n-1}})^\tr$.
There is an obvious embedding of $K$ into $K^\tr$, and hence into $K^{\>\tr^{\<n}}\<$. 
For a map 
\mbox{$p\colon K\to\sC$} of \ics, the corresponding \emph{undercategory} $\sC_{p/}$ is a simplicial set whose $(n-1)$-simplices ($n>0$)
are the extensions of $p$ to maps $K^{\>\tr^{\<n}}\to \sC$, see \cite[p.\,43, 1.2.9.5]{T}. This undercategory is an \ic\ \cite[p.\,61, 2.1.2.2]{T}.

\begin{subdef} A \emph{colimit} of a map $p\colon K\to\sC$ of \ics\ is an initial object in
the \ic\ $\sC_{p/}$.
\end{subdef}

Being an object (= 0\kf-simplex) in $\sC_{p/}\>$, any colimit of $p$ is an extension of~$p$ to 
a map $\bar p\colon K^\tr\to\sC$. Often one refers loosely to the image under $\bar p$ of the vertex $*_0\in K^\tr$ as the colimit of $p$.

\smallskip
Some instances of colimits are the $\infty$-categorical versions of \emph{coproducts} (where $K$ is the nerve of a category whose only maps are identity maps),
\emph{pushouts} (where $K$ is the horn $\Lambda_0^2\>$),  and \emph{coequalizers}
(where $K$ is the nerve of a category with exactly two objects $x_1$ and $x_2$, and such that $\Map{}(x_i,x_j)$ has cardinality $j-i+1$), see \cite[\S4.4]{T}.

\begin{ex} Suppose $\sC$ is the nerve $\N(C)$ of an ordinary category~$C$. A functor
$p\colon K\to \sC$ corresponds under the adjunction $\h\dashv \N$ to a functor 
$\tilde p\colon \h K\to C$. There is a natural isomorphism of ordinary categories
\[\h\,(K^\tr)\cong (\h K)^\tr,
\]
whence an extension of $p$ to $K^\tr$ corresponds under $\h\dashv \N$ to 
an extension of\kf~$\tilde p$ to $(\h K)^\tr$. More generally, one checks that there is a natural isomorphism 
\[
\sC_{p/} = \N(C)_{p/}\cong\N(C_{\tilde p/}).
\]
Any colimit of $p$ is an initial object in 
$\h\sC_{p/} \cong\h \N(C_{\tilde p/})\cong C_{\tilde p/}\,$; that is,\va{.6} \emph{the homotopy functor takes a colimit of\/ $p$ to a colimit of} $\tilde p$.

\smallskip
For more general $\sC$, and most $p$, the homotopy functor does not preserve colimits.
For example, in any \emph{stable} \ic, like the derived \ic\ of a Grothendieck abelian category
\cite[p.\,96, 1.3.5.9]{A}, the pushout of 0 with itself over an object $X$ is the suspension $X[1]$
(see \cite[p.\,19, bottom paragraph]{A}), but the pushout in the homotopy category is 0.

 \end{ex}

\section{Adjoint Functors}
For a pair of functors (=\:simplicial maps) $\sC\xto{\>f\>}\sD\xto{\>g\>}\sC$ of \ics\ 
one says that $f$ \emph{is a left adjoint of}~$g$, or that $g$~\emph{is a right adjoint of}~$f\<$, if there exists a 
homotopy $u$ from the identity functor $\id_\sC$ to $gf$ (that is, a  simplicial map
$u\colon\sC\times \Delta^{\<1}\to\sC$ whose compositions with the maps
\[
\sC=\sC\times\Delta^{\<0}\xto{\{i\}}\sC\times\Delta^{\<1}\qquad(i=0,1)
\]
corresponding to the 0\kf-simplices $\{0\}$ and $\{1\}$ of $\Delta^{\<1}$ are $\id_\sC$ and $gf$, respectively) such that, for all objects $C\in\sC$ and $D\in\sD$, the natural composition 
\[
\Map{\sD}(f(C), D)\to\Map{\sC}\big(gf(C),g(D)\big)\xto{\!u(C)\>} \Map{\sC}(C, g(D))
\]
is an isomorphism in $\sH$.\va1

(For an extensive discussion of adjunction, see \cite[\S5.2]{T}.  The foregoing definition
comes from  \cite[p.\,340, 5.2.2.8]{T}.)\smallbreak

Such adjoint functors $f$ and $g$ induce adjoint functors 
\smash{$\h\sC\xto{\>\h f\>}\h\sD\xto{\>\h g\>}\h\sC$}
between the respective $\sH$-enriched homotopy categories.

As a partial converse, it holds that if the functor $\h f$ induced by a functor $f\colon\sC\to\sD$
between \ics\ has an $\sH$-enriched right adjoint, then $f$~itself has a right adjoint \cite[p.\,342, 5.2.2.12]{T}.

\smallskip The following \emph{Adjoint Functor Theorem} gives a powerful criterion 
(to be used subsequently) for 
$f\colon \sC\to\sD$ to have a right adjoint. It  requires 
a restriction---\emph{accessibility}---on the sizes of the \ics\ $\sC$ and $\sD$. This means roughly that $\sC$ is generated under filtered colimits by a small
$\infty$-subcategory, and similarly for~$\sD$, see \cite[Chap.\,5]{T}. 
(If necessary, see also \cite[p.\,51]{T} for the explication of ``small" in the context of Grothendieck universes.) Also, $\sC$ and 
$\sD$ need to admit colimits of all maps they receive from small simplicial sets $K\<$.  The conjunction of these properties is called \emph{presentability} \cite[\S5.5]{T}.\looseness=-1   

For example, the \ic\ $\sS$ of spaces (see \ref{spaces}) is presentable \cite[p.\,460, 5.5.1.8]{T}. It follows that the \ic\ of spectra 
Sp $\set\textup{Sp}(\sS_*)$ (see \cite[p.\,116, 1.4.2.5]{A} and \cite[p.\,122, 1.4.3.1]{A})
is presentable. Indeed, presentability is an equivalence-invariant property of \ics\kern-.5pt,
see e.g.,  \cite[p.\,457, 5.5.1.1(4)]{T}, hence by the presentability of $\sS$ and by \cite[p.\,719; 7.2.2.8]{T}, \cite[p.\,242; 4.2.1.5]{T} and \cite[p.\,468, 5.5.3.11]{T}, 
$\sS_*$  is presentable, whence, by \cite[p.\,127, 1.4.4.4]{A}, so is Sp. 

\begin{thm} \textup{(\cite[p.\,465, 5.5.2.9]{T}.)}\label{Brown} 
A functor\/ $f\colon \sC\to\sD$ between presentable \ics\
has a right adjoint iff it preserves small colimits.
\end{thm}

\begin{cosa}\label{Fun}
Let $\sC$ and $\sD$ be \ics. The
simplicial set $\mathbf{Hom}(\sC, \sD)$ \cite[\S I.5]{GJ} is an \ic,  denoted
$\Fun \sC \sD$ \cite[p.\,39, 1.2.7.3]{T}.  Its 0\kf-simplices are functors (=\:simplicial maps). Its 1-simplices are maps of functors: such a map $f\to g$ is, by definition, a simplicial map $\phi\colon \sC\times\Delta^{\<1}$ to~$\sD$ such that the following functor is $f$ when $i=0$ and $g$ when $i=1$:
\[
\sC=\sC\times\Delta^{\<0}\xto{\<\id\<\times\> i\>}\sC\times\Delta^{\<1}\xto{\,\phi\,}\sD.
\]

Let $\textup{Fun}^{\textup L}(\sC,\sD)$ (resp.~$\textup{Fun}^{\textup R}(\sC,\sD))$ be the full $\infty$-subcategories  spanned by the functors which are left (resp.~right) adjoints, that is, the \ics\ whose simplices
are all those in~$\Fun \sC\sD$ whose vertices are such functors. 

The \emph{opposite} $\sE^\op$ of an \ic\ $\sE$ \cite[\S\kf1.2.1]{T} is the simplicial set having the same
set $\sE_n$ of $n$-simplices as $\sE$ for all $n\ge0$, but with face and degeneracy operators
\begin{align*}
(d_i\colon \sE_n^\op\to \sE_{n-1}^\op) &\set (d_{n-i}\colon\sE_n\to\sE_{n-1}),\\
(s_i\colon \sE_n^\op\to \sE_{n+1}^\op) &\set (s_{n-i}\colon\sE_n\to\sE_{n+1}).
\end{align*}
It is immediate that $\sE^\op$ is also an \ic.\va1

The next result, when restricted to ordinary categories, underlies the notion of \emph{conjugate functors} (see, e.g., \cite[3.3.5--3.3.7\kf\kf]{li}.)

\enlargethispage*{5pt}
\begin{prop}\label{conjugacy}
There is  a canonical $(\<$up to homotopy$\>)$ equivalence
\begin{equation}\label{left-right}
\varphi\colon\textup{Fun}^{\textup R}(\sC,\sD)\eqv\textup{Fun}^{\textup L} (\sD,\sC)^\op.
\end{equation}
that takes any object $g\colon \sC\to \sD$ in\/ $\textup{Fun}^{\textup R} (\sC,\sD)$ 
to a left-adjoint functor\/ $g'\<$. 
\end{prop}
\end{cosa}

\section{Algebra objects in monoidal \ics}

A \emph{monoidal category} $\sM$ is a category together with
a monoidal structure, i.e., a \emph{product functor}
$\otimes\colon \sM\times \sM\to \sM$ that is associative up to isomorphism,
plus a \emph{unit object} $\cO$ and isomorphisms (unit maps)
\[\cO\otimes M\iso M\iso M\otimes \cO\qquad(M\in \sM)
\]
compatible with the associativity isomorphisms.

An \emph{associative algebra} $A\in \sM$ ($\sM$-algebra for short) is an object equipped with maps $A\otimes A\to A$ (multi\-plication)
and $\cO\to A$ (unit) satisfying associativity etc. \emph{up~to isomorphism}, such isomorphisms having the usual relations, expressed by commutative diagrams.

(No additive structure appears here, so one might be tempted to call algebras ``monoids." 
However, that term is reserved in \cite[\S2.4.2]{A} for a related, but different, construct.)

 \begin{exs} \label{monoids} 
 (a) $\sM\set{}$\{Sets\}, $\otimes$ is the usual direct product, and $\sM$-algebras are  monoids.\va2

 (b) $\sM\set\:$modules over a fixed commutative ring $\cO$, $\otimes$ is the usual tensor product over $\cO$, and $\sM$-algebras are  the usual $\cO$-algebras.\va2

(c) $\sM\set\:$dg modules over a fixed commutative dg ring $\cO$, $\otimes$ is the usual tensor product of dg $\cO$-modules,  and an $\sM$-algebra is a  dg $\cO$-algebra (i.e., a dg ring $A$ plus
a homomorphism of dg rings from $\cO$ to the center of $A$). \va2

 (d) $\sM\set{}$ the derived category $\D(X)$ of $\cO$-modules over a  (commutative) ringed space
 $(X\<,\cO)$, $\otimes$ is the derived tensor product of $\cO$-complexes. Any dg $\cO$-algebra gives rise to an $\sM$-algebra; but there might be $\sM$-algebras not of this kind, as the
defining diagrams may now involve quasi-isomorphisms and homotopies, not just equalities. 

 \end{exs}

The foregoing notions can be extended to \ics. 
The key is to formulate how algebraic structures in categories arise from \emph{operads,} in a way that can be upgraded to \ics\ and $\infty$-operads. Details of the actual implementation are not effortless to absorb. (See \cite[\S4.1, etc.]{A}.)

The effect is to replace isomorphism by ``coherent homotopy."
Whatever this means (see \cite[\S\kf1.2.6]{T}), it turns out that any monoidal structure on 
an \ic\  $\sC$ induces a monoidal structure on the ordinary category~$\h \sC$,
and any $\sC$-algebra (very roughly: an object with multiplication associative 
up to coherent homotopy) is taken by the homotopy functor to an $\h \sC$-algebra \cite[p.\,332, 4.1.1.12, 4.1.1.13]{A}.

The $\sC$-algebras  are the objects of an \ic\ $\Alg(\sC)$ \cite[p.\,331, 4.1.1.6]{A}. The point is that the 
the homotopy-coherence of the associativity and unit
maps are captured by an \ic\ superstructure. 
 
\smallskip
Similar remarks apply to \emph{commutative} $\sC$-algebras, that is,  $\sC$-algebras whose multiplication is commutative up to coherent homotopy.

 \begin{ex} In the  monoidal \ic\ of spectra \cite[\S\S\kf1.4.3, 6.3.2]{A},  algebras are called  $A_\infty$-rings, or $A_\infty$-ring spectra; and commutative algebras are called  
$\mathbb E_\infty$-rings, or $\mathbb E_\infty$-ring spectra. The \emph{discrete} $A_\infty$-(resp.~$\mathbb E_\infty$-) rings---those algebras $S$ whose homotopy groups $\pi_iS$ vanish for $i\ne 0$---span an 
\ic\ that is equivalent to (the nerve of) the category of associative (resp.~commutative) rings
\cite[p.\,806, 8.1.0.3]{A}. 

In general, for any  commutative ring~$R$, there is a close relation between dg $R$-algebras and $A_\infty$-$R$-algebras, see \cite[p.\,824, 8.1.4.6]{A}, \cite[Thm.\,1.1]{Sh2}; 
and  when $R$ contains the rational field $\mathbb Q$, between 
graded-commutative dg $R$-algebras and $\mathbb E_\infty$-$R$-algebras 
\cite[p.\,825, 8.1.4.11]{A}; in such situations, every $A_\infty$-(resp.~$\mathbb E_\infty$-)$R$-algebra is equivalent to a dg $R$-algebra.\va2

See also \cite[\S4.1.4]{A} for more examples of \ic-algebras that have concrete representatives.

 \end{ex}

\section{Bimodules, tensor product}

\begin{cosa}
For algebra objects $A$ and $B$  in  a  monoidal \ic\ $\sC$,  there is a notion of 
\emph{$A\>$-$B$-bimodule}---an object in $\sC$ on which, via $\otimes$ product in~$\sC$, $A$~acts on the left, $B$ on the right and the actions commute up to coherent homotopy. (No additive structure is required.) 
The  bimodules in $\sC$ are the objects of an \ic\ $_A\textup{BMod}_B(\sC)$,
to be denoted here, once the \ic\ $\sC$ is fixed, as $\Bmod AB$.
(See \cite[\S4.3]{A}.)
\end{cosa}

\begin{cosa}\label{infinity tensor}

Let $A$, $B$, $C$ be algebras in a monoidal \ic\ $\sC$  that  admits small colimits, and in which product functor $\sC\times\sC\to \sC$ preserves small colimits separately in each variable. There is a tensor-product functor
\begin{equation}\label{inftens}
{(\Bmod AB)\times (\Bmod BC)\xto{\:\otimes\:}\Bmod AC,}
\end{equation}
defined to be the \emph{geometric realization}---a kind of colimit, see  \cite[p.\,542, 6.1.2.12]{T}---of a two-sided bar-construction.
(See \cite[p.\,409, 4.3.5.11]{A}; and
for some more motivation, \cite[pp.145--146]{Lu0}.)\va2

Tensor product is associative up to canonical homotopy \cite[p.\,416, 4.3.6.14]{A}. It is unital on the left in the sense that, roughly, the endofunctor of $\Bmod BC$ given by tensoring on the left with the $B$-$B$-bimodule $B$ is canonically homotopic to the identity; and similarly on the
right \cite[p.\,417, 4.3.6.16]{A}.
Also, it preserves colimits separately in each variable \cite[p.\,411, 4.3.5.15]{A}.\va2

\begin{ex}[Musings] \label{derived tensor} It is natural to ask about direct connections between \eqref{inftens} and the usual tensor product of bimodules over rings. If there is an explicit answer in the literature I haven't found it, except when $A=B=C$ is an ordinary commutative ring
regarded as a discrete $\mathbb E_\infty$-ring, a case addressed 
by \cite[p.\,817, 8.1.2.13]{A} (whose proof might possibly be adaptable to a more general situation). 

What follows are some related remarks, in the language of model categories, which in the present context can presumably be translated into the language of \ics. \va1(Cf.~e.g., 
\cite[p.\,90, 1.3.4.21, and p.\,824, 8.1.4.6]{A}.)

Let $\sM$ be the (ordinary) symmetric monoidal category of abelian groups, and let 
$A$, $B$, $C$ be $\sM$-algebras, i.e., ordinary rings (set $\cO\set\ZZ$ in~\ref{monoids}(b)). The tensor product over $B$ of an $A\otimes B^\op$-complex 
(i.e., a left $A$- right $B$-complex, or $A$-$B$-bicomplex) and a 
$B\otimes C^\op$-complex is an $A\otimes C^\op$-complex.

\noindent \emph{Can this bifunctor be extended to a derived functor}
\[ 
\D(A\otimes B^\op)\times \D(B\otimes C^\op) \to \D(A\otimes C^\op)?
\]

To deal with the question it seems necessary to move out into the dg~world. Enlarge $\sM$ to the category of complexes of abelian groups, made into a symmetric monoidal model category
by the usual tensor product and the ``projective" model structure (weak equivalences
being quasi-isomorphisms, and fibrations being surjections),  see \cite[p.\,816, 8.1.2.11]{A}.
It results from \cite[Thm.\,4.1(1)]{SS} (for whose hypotheses see \cite[p.\,356, \S2.2 and p.\,359, Proposition 2.9]{Sh2}) that: \va1

1) For any $\sM$-algebra (i.e., dg ring) $S$,  the category~$\sM_S\subset\sM$ of left dg $S$-modules has a model structure for which maps are weak equivalences (resp.~fibrations) if and only if they are so in $\sM$.\va1

2) For any commutative $\sM$-algebra (i.e., graded-commutative dg ring)~$R$, the category of $R$-algebras in $\sM$ has a model structure for which maps are weak equivalences (resp.~fibrations) if and only if they are so in $\sM$.\va1

In either case 1) or 2), the  \emph{cofibrant}\va{.7} objects are those~$I$ such~that for any diagram $I\xto{\lift1.1,\>f,}Y\xot{\lift.5,\>\>s\>,}X$ with $s$ a\va{.3}  surjective quasi-isomorphism, there exists (in the category in play) $g\colon I\to X $ such that $sg=f$. Any object $Z$ in these model categories\va{.3} is the target of a quasi-isomorphism 
$\tilde Z\to Z$ with cofibrant~$\tilde Z$; such a quasi-isomorphism (or its source) is called a \emph{cofibrant replacement} of~$Z$. \va1

Note that the the derived category $\D(S)$---obtained by adjoining 
to~$\sM_S$ formal inverses of its quasi-isomorphisms---is the homotopy category
of the model category $\sM_S$.\va1

Now fix a graded-commutative dg ring~$R$. The \emph{derived tensor product} $S\Otimes{R} T$ of two dg $R$-algebras $S$, $T$ is the tensor product $\tilde S\otimes_R\tilde T$.%
\footnote{This is an instance of the
passage from a monoidal structure on a model category~ 
$\sM$ to~one on the homotopy category of $\sM$ \cite[\S4.3]{Ho}---a precursor of the
passage from a monoidal structure on an \ic\  
to~one on its homotopy category.}
This construction depends, up to quasi-isomorphism, on the choice of the cofibrant replacements.
However, two such derived tensor products have canonically isomorphic derived categories
\cite[Theorem 4.3]{SS}. Any such derived category will be denoted $\D(S\Otimes{R}T)$.

If either $S$ or $T$  is \emph{flat} over $R$ then the natural map $S\Otimes{R}T\to S\otimes_RT$ is a quasi-isomorphism; in this case one need not distinguish between the derived and the ordinary  tensor product.

More generally, let $S$ and $T$ be dg $R$-algebras, let $M$ be a dg $S$-module 
and $N$ a dg $T$-module. Let $\tilde S\to S$ and $\tilde T\to T$ be cofibrant replacements.
Let $\tilde M\to M$ (resp.~$\tilde N\to N$) be a cofibrant replacement in the category of dg 
$\tilde S$-(resp.~$\tilde T$-)modules. Then $\tilde M\otimes_R\tilde N$ is a dg module over 
$\tilde S\otimes_R \tilde T$. Using ``functorial factorizations"  \cite[Defn.\,1.1.3]{Ho}, one finds that this association of 
$\tilde M\otimes_R\tilde N$ to  $(M,N)$ gives rise to a functor 
\[
\D(S)\times\D(T)\to \D(S\Otimes{R}T),
\]
and that different choices of cofibrant replacements lead canonically to isomorphic functors.

If $A$, $B$ and $C$ are dg $R$-algebras, with $B$ commutative, then 
setting $S\set A\otimes_RB$ and
$T\set B\otimes_RC^\op$, 
one gets, as above, a functor
\[
\D(A\otimes_R B)\times\D(B\otimes_R C^\op)\to \D\big((A\otimes_{R}B)\Otimes{B}(B\otimes_RC^\op)\big). 
\]
Then, via restriction of scalars through the natural map 
\[\D(A\Otimes{R} C^\op)\to\D\big((A\otimes_{R}B^\op)\Otimes{B}(B\otimes_RC^\op)\big)
\]
one gets a version of the desired functor, of the form
 \[
\D(A\otimes_R B)\times\D(B\otimes_R C^\op)\to \D(A\Otimes{R}C^\op). 
\]

What does this functor have to do with the tensor product of \S\ref{infinity tensor}? 

Here is an approach that should lead to an answer;  but details 
need to be worked out.

Restrict $R$ to be an ordinary commutative ring, and again, $B$ to be graded-commutative. 
By \cite[2.15]{Sh2} there is a zig-zag $\mathbb H$ of three
``weak monoidal Quillen equivalences" between the model category of dg $R$-modules (i.e., $R$-complexes) and the model category of symmetric module spectra over the Eilenberg-Mac\,Lane symmetric spectrum~$HR$ (see e.g., \cite[4.16]{Gr}), that induces\-
a monoidal equivalence between the respective homotopy categories. 
(The monoidal structures on the model categories are given by the tensor and smash products, respectively---see e.g., \cite[Chapter 1, Thm.\,5.10]{Sc}. 
For~$A$-$B$-bimodules~$M$
and $B$-$C$-bimodules~$N$, the tensor product \mbox{$M\!\otimes_B\! N$} coequalizes the natural maps 
\mbox{$M\otimes_R B\otimes_R N\:\lift 1.2,\overrightarrow{\lto}, \:M\otimes_R N$,} and likewise for the smash product of $H M$ and~$H N$ over~$HB$; so it should follow that when $M$ and $N$ are cofibrant,
 these products also correspond, up to homotopy, under $\mathbb H$. This would reduce the problem, modulo homotopy, to a comparison of the smash product and the relative tensor product in the associated \ic\ of the latter category. But it results from \cite[4.9.1]{Sh1} that these bifunctors become naturally isomorphic in the 
homotopy category of symmetric spectra, i.e., the classical stable homotopy category.\va1

For a parallel approach, based on ``sphere-spectrum-modules" rather than symmetric spectra,
see \cite[\S IV.2 and Prop.\,IX.2.3]{EKMM}.\va1

Roughly speaking, then, any homotopical---i.e., equivalence-invariant---property of relative tensor products in the \ic\  of spectra (whose homotopy category is the stable homotopy category) should entail a  property of derived tensor products of dg modules or bimodules over appropriately commutative (or not) dg $R$-algebras.\looseness=1
\end{ex}
\end{cosa}


\section{The $\infty$-functor $\sHom$}

One shows, utilizing \cite[p.\,391, 4.3.3.10]{A}, that if $\sC$ is presentable then  so~is $\Bmod AB$.
Then one can apply the Adjoint Functor Theorem~\ref{Brown} to prove:

\begin{prop}\label{sHom}
Let\/ $A,B,C$ be algebras in a fixed presentable monoidal \ic. There exists a functor
\[
\sHom_{\>C}\colon (\Bmod B{\>C})^\op\times \Bmod AC\to \Bmod AB
\]
such that for every fixed\/ $y\in\Bmod BC,$ the functor
\[
z\mapsto\sHom_{\>C}(y,z)\colon \Bmod AC\to\Bmod AB
\] 
is right-adjoint to the functor\/ $x\mapsto x\otimes_B y\colon \Bmod AB\to\Bmod AC$.
\end{prop}

As adjoint functors between \ics\ induce adjoint functors between the respective homotopy categories, and by unitality of tensor product,  when $x=A=B$ one gets:

\begin{cor}[``global sections" of $\sHom\;$= mapping space]\label{Hom+Map} There exists
an\/ $\sH$-isomorphism of\/  functors $($going from 
$\h\, (\Bmod AC)^\op\times \h\,( \Bmod AC)$ to $\sH\>$$)$
\[
\Map{\<\Bmod AA}(A, \sHom_{\>C}(y,z))\iso \Map{\<\Bmod AC}(A\otimes_A y, z)\cong\Map{\<\Bmod AC}(y, z).
\]
\end{cor}

\pagebreak[3]

\section{Adjoint Associativity in \ics}

We are finally in a position to make sense of adjoint associativity for \ics. The result and proof are similar in spirit to, if not implied by, those  in \cite[p.\,358, 4.2.1.31 and 4.2.1.33(2)]{A}
about ``morphism objects."

\begin{thm}\label{adjassoc}
There is  in\/ \textup{Fun}$((\Bmod AB)^\op\times(\Bmod BC)^\op\times \Bmod DC,\, \Bmod DA)$ a
functorial equivalence $($canonically defined, up to homotopy$)$
\[
\alpha(x,y,z)\colon\sHom_{\>C}(x\otimes_B\< y, z)\to \sHom_B(x,\sHom_{\>C}(y,z))
\]
such that for any objects\/ $x\in\Bmod AB$ and $y\in\Bmod BC\>,$ 
the map\/ $\alpha(x,y,-)$ in $\textup{Fun}^{\textup{R}}(\Bmod DC ,\Bmod DA )$ is taken by the  the equivalence\/ \textup{\ref{left-right}} to the associativity equivalence,  in\/ 
$\textup{Fun}^{\textup{L}}(\Bmod DA ,\Bmod DC )^\op,$\va{-1}
\[
-\otimes_A (x\otimes_B\< y)\leftarrow (-\otimes_A x)\otimes_B y.
\]
\end{thm}

Using Corollary~\ref{Hom+Map}, one deduces:
\begin{cor}\label{global assoc}
In  the homotopy category of spaces there is a trifunctorial isomorphism\va{-1}
\[
\Map {\<\Bmod {\<A}C}(x\otimes_B\< y,z)\iso \Map{\<\Bmod AB}(x,\sHom_{\>C}(y,z))\\[3pt]
\]
\rightline{$(x\in\Bmod AB,\,y\in\Bmod BC,\,z\in\Bmod AC).$}
\end{cor}

\begin{ex}[More musings]\label{derived hom} 
What conclusions about ordinary algebra can we draw? 

Let us confine attention to spectra, and try to understand the homotopy invariants 
of the mapping spaces in the preceding Corollary, in particular the maps in the corresponding unenriched homotopy categories (see last paragraph in \S\ref{map space}).

As in Example~\ref{derived tensor}, the following remarks outline a possible approach, 
whose details I have not completely verified.

Let $R$ be an ordinary commutative ring. Let $S$ and $T$ be dg $R$-algebras, and $U$ a derived tensor product $U\set S\Otimes{R} T^\op$ (see Example~\ref{derived tensor}). 
For any dg~$R$-module $V\<$, let $\widehat V$ be the canonical image of  $\mathbb HV$ ($\mathbb H$ as in~\ref{derived tensor})
 in the associated \ic~$\sD$ of the model category $\mathbf A$ of  $HR$-modules. 
 The \ic~$\sD$ is monoidal via a suitable extension, denoted $\wedge$, of the smash product, see \cite[p.\,619, (S1)]{A}. Recall from the second-last paragraph in~\ref{model cat} that the homotopy
 category $\h \sD$ is equivalent to the homotopy category of~$\mathbf A$.

 As indicated toward the end of Example~\ref{derived tensor}, \emph{there should~be,} in~$\sD$,
 an~equivalence\looseness=-1
\[
\widehat U\overset{\lift.5,\approx,}\lto\widehat S\wedge\widehat{T^\op}. 
\]
whence, by \cite[p.\,650, 6.3.6.12]{A}, an equivalence
\[
\Bmod {\widehat S}{\widehat T}\simeq\lmod {\widehat U}\>,
\]
whence, for any dg $S$-$T$ bimodules $a$ and $b$,  and $i\in\ZZ$, isomorphisms in the  homotopy category $\sH$ of spaces
\begin{equation}\label{pi1}
\Map {\<\Bmod{\widehat S}{\widehat T}}(\widehat a,\widehat{b[i]})
\iso \Map{\lmod{\widehat U}}(\widehat a,\widehat{b[i]}).
\end{equation}
(For the hypotheses of \emph{loc.\,cit.}, note that the \ic\ Sp of spectra, being presentable, has small 
colimits---see remarks preceding Theorem~\ref{Brown}; and these colimits are preserved by smash product \cite[p.\,623, 6.3.2.19]{A}.)\va1

By \cite[p.\,393, 4.3.3.17\kf\kf]{A} and again, \cite[p.\,650, 6.3.6.12]{A}, the stable \ic\ 
$\lmod{\widehat U}$ is equivalent to the associated \ic\ of the model category of 
$ H U$-module spectra, and hence to the associated \ic~$\sU$ of the equivalent
model category of dg $U$-modules. There results an $\sH$-isomorphism
\begin{equation}\label{pi2}
\Map{\lmod{\widehat U}}(\widehat a,\widehat{b[i]})\iso \Map{\>\sU\<}(a,b[i]).
\end{equation} 

Since the homotopy category $\h\lmod{\widehat U}$ is equivalent to $\h\sU\set \D(U)$, 
\eqref{pi1} and~\eqref{pi2} give isomorphisms, with $\ext^i_{\Bmod{\widehat S}{\widehat T}}$ 
as in~\cite[p.\,24, 1.1.2.17]{A}:  
\[
\ext^i_{\Bmod{\widehat S}{\widehat T}}(\widehat a,\widehat{b})\set 
\pi_0\Map {\<\Bmod{\widehat S}{\widehat T}}(\widehat a,\widehat{b[i]})
\iso \Hom_{\D(U)}(a,{b[i]})=\ext^i_U(a, b).
\]
(In particular, when $S=T=a$, one gets the derived Hochschild cohomology of~$S/R$, with coefficients in $b$.)\va1

Thus, Corollary~\ref{global assoc} implies a derived version, involving Exts, of adjoint associativity for
dg bimodules. 
\end{ex}

\pagebreak[3]

\begin{proof}[Proof of Theorem \textup{\ref{adjassoc} (Sketch)}]

The associativity of tensor product 
 gives a canonical equivalence, in $\Fun{(\Bmod DA)\times(\Bmod AB)\times(\Bmod BC)}{(\Bmod DC)}$,  between the composed  functors
\begin{align*}
&(\Bmod DA)\times(\Bmod AB)\times(\Bmod BC)
\xto{\!\otimes\times\id\,}
(\Bmod DB)\times(\Bmod BC)
\xto{\,\otimes\,}(\Bmod DC),\\
&(\Bmod DA)\times(\Bmod AB)\times(\Bmod BC)
\xto{\!\id\<\times\>\otimes\,}
(\Bmod DA)\times(\Bmod AC)
\xto{\,\otimes\,}(\Bmod DC).
\end{align*}
The standard isomorphism $\Fun{X\times Y}Z\iso \Fun{X}{\Fun YZ}$ 
(see \cite[p.\,20, Prop.\,5.1]{GJ}) turns this into an equivalence $\xi$ between the corresponding functors from 
$(\Bmod AB)\times (\Bmod BC)$ to $\Fun{\Bmod DA}{\Bmod DC}$. These functors factor through 
the full subcategory $\textup{Fun}^{\textup L}(\Bmod DA,\Bmod DC)$: this need only be checked at the level of objects $(x,y)\in (\Bmod AB)\times(\Bmod BC)$, whose image functors are, by Proposition~\ref{sHom}, left-adjoint, respectively, to 
 $\sHom_B(x,\sHom_C(y,-))$ and to~$\sHom_C(x\otimes y,-)$. Composition with $(\varphi^{-1})^\op$ ($\varphi$ as in \eqref{left-right}) takes $\xi$
into an equivalence  in 
\begin{multline*}
\Fun{(\Bmod AB)\times (\Bmod BC)}{{\textup{Fun}^{\textup R}(\Bmod DC,\Bmod DA)}^\op\>}\\
=\Fun{(\Bmod AB)^\op\times (\Bmod BC)^\op}{{\textup{Fun}^{\textup R}(\Bmod DC,\Bmod DA)}},
\end{multline*}
to which $\alpha$ corresponds.
(More explicitly, note that for any \ics\ $X$, $Y$ and $Z$, there is a \emph{composition functor}
\begin{equation}\label{comp}
\Fun YZ\times\Fun XY\to\Fun XZ
\end{equation}
corresponding to the natural composed functor
\[
\Fun YZ\times \Fun XY\times X \to \Fun YZ\times Y\to Z\>;
\]
and then set $X\set (\Bmod AB)^\op\times (\Bmod BC)^\op$, 
$Y\set \textup{Fun}^{\textup L}(\Bmod DA,\Bmod DC)^\op$
and \mbox{$Z\set \textup{Fun}^{\textup R}(\Bmod DC,\Bmod DA)$}\dots)

The rest follows in a straightforward manner from Proposition~\ref{conjugacy}.
\end{proof}

\bigskip

\textbf{Acknowledgement} I am grateful to Jacob Lurie
and Brooke Shipley for enlightenment on a number of questions that arose during the preparation of this paper.


\begin{thebibliography}{EKMM}

\bibitem[AI]{AI} L.\,L.\ Avramov, S.\,B.\  Iyengar,  Gorenstein algebras and Hochschild cohomology. {\it Michigan Math. J.} {\bf 57} (2008), 17--35.

\bibitem[AIL]{AIL} \bysame, J.\ Lipman,  Reflexivity and rigidity for complexes, 
   II: Schemes. {\it Algebra  Number Theory} {\bf5} (2011), 379--429. 


\bibitem[AILN]{AILN} \bysame, J.\ Lipman, S.\ Nayak, Reduction of derived Hochschild functors over commutative algebras and schemes. {\it Advances in Math.} {\bf 223} (2010), 735--772.

\bibitem[BF]{BF} R.\ Buchweitz, H.\ Flenner, 
Global Hochschild (co\kf-)homology of singular spaces.  \textit{Advances in Math.} \textbf{217} (2008), 205--242.

\bibitem[EKMM]{EKMM} A.\,D.\ Elmendorf, I.\ Kriz,  M.\,A.\ Mandell, J.\,P.\ May, \textit{Rings, modules and algebras in stable homotopy theory.} Mathematical Surveys and Monographs 47. American Mathematical Society, Providence, RI, 1997.

\bibitem[ILN]{ILN}
S.\,B.\ Iyengar, J.\ Lipman, A.\,Neeman, Relation between two twisted inverse image pseudofunctors in duality theory. \texttt{arXiv:1307.7092}.

\bibitem[GJ]{GJ} P.\,G.\ Goerss,  J.\,F.\ Jardine, \textit{Simplicial homotopy theory.} Progress in Mathematics, No.\ 174. Birkh\" auser Verlag, Basel, 1999.

\bibitem[Gr]{Gr} J.\,P.\,C.\ Greenlees,  Spectra for commutative algebraists. \textit{Interactions between homotopy theory and algebra.} Contemp. Math., 436, Amer. Math. Soc., Providence, RI, 2007, 149--173.

\bibitem[H]{RD} R.\,Hartshorne,  \textit{Residues and Duality.} Lecture
Notes in Math., \textbf{20}. Springer-Verlag, Berlin-New York, 1966.

\bibitem[HK]{HK} H.\,S.\,Hall, S.\,R.\,Knight, \textit{Higher\:Algebra,} 4th edition.
MacMillan and Co. Ltd., \mbox{London, 1907}.

\bibitem[Ho]{Ho} M.\,Hovey, Model categories. 
Mathematical Surveys and Monographs, \textbf{63} . American Mathematical Society, Providence, RI, 1999. 

\bibitem[L]{li} J.\,Lipman, Notes on derived categories and Grothendieck Duality. \textit{Foundations of Grothendieck duality for diagrams of schemes.}  Lecture Notes in Math., {\bf 1960}, Springer-Verlag, Berlin-New York, 2009, 1--259.

\bibitem[Lu0]{Lu0} Lurie, J., Derived Algebraic Geometry II: Noncommutative Algebra.\newline
\texttt{arXiv:math/0702299v5}.

\bibitem[Lu1]{T} \bysame, \textit{Higher Topos Theory.} Annals of Mathematics Studies, No.~170. Princeton University Press, Princeton, NJ, 2009. Manuscript downloadable from 
\texttt{www.math.harvard.edu/\~{}lurie/ }

\bibitem[Lu2]{A} \bysame, \textit{Higher Algebra} (Aug., 2012). 
\texttt {\ www.math.harvard.edu/\~{}lurie/ }


\bibitem[M]{M} S.\,MacLane, 
\emph{Homology}.
Grundlehren Math. Wiss. 114. Springer-Verlag,  Berlin-New York, 1967.

\bibitem[Sc]{Sc}  S.\,Schwede, Symmetric Spectra.\newline
{\tt www.math.uni-bonn.de/\~{}schwede/SymSpec-v3.pdf}\va1


\bibitem[Sh1]{Sh1} B.\,Shipley, Symmetric spectra and topological Hochschild homology. K-Theory \textbf{19} (2000), 155--183. 

\bibitem[Sh2]{Sh2} \bysame,  $H\mathbb Z\kf$-algebra spectra are differential graded algebras. 
{\it Amer.~J.~Math.} \textbf{129} (2007),  351--379.

\bibitem[SS]{SS} \bysame, S.\,Schwede,  Algebras and modules in monoidal model categories. 
{\it Proc.~London Math.~Soc.} (3) \textbf{80}, (2000), 491--511. 

%
\bibitem[Y]{Y} A.\ Yekutieli, Rigid dualizing complexes via differential graded algebras (survey). \emph{Triangulated categories,} 452--463, London Math. Soc. Lecture Note Ser., 375. Cambridge Univ. Press, Cambridge, 2010.

\end{thebibliography}
\end{document}